\documentclass{amsart}

\usepackage{geometry}
\usepackage{amssymb}
\usepackage{url}
\usepackage{amsfonts}
\usepackage[dvipsnames]{xcolor}
\usepackage{amsthm}
\usepackage{graphicx}
\usepackage{float}
\usepackage{amsmath}

\theoremstyle{plain}

\newtheorem{theorem}{Theorem}[section]

\theoremstyle{definition}

\newtheorem{remark}[theorem]{Remark}

\newcommand{\Z}{{\mathbb Z}}
\newcommand{\R}{{\mathbb R}}

\numberwithin{equation}{section}

\usepackage{lineno}

\begin{document}

\title[Orlicz-Lorentz Gauge Functional Inequalities for Positive Integral Operators (Revised Version)]{Orlicz-Lorentz Gauge Functional Inequalities for Positive Integral Operators\\ (Revised Version)}

\author[R.~Kerman]{Ron Kerman}
\address{R.~Kerman\\Department of Mathematics and Statistics,
Brock University, 1812 Sir Isaac Brock Way, St. Catharines, ON L2S 3A1
}
\email{rkerman@brocku.ca}

\author[S.~Spektor]{Susanna Spektor}
\address{S.~Spektor\\Department of Mathematics and Statistics
Sciences, PSB,
Sheridan College Institute of Technology and Advanced Learning,
4180 Duke of York Blvd., Mississauga, ON L5B 0G5}
\email{susanna.spektor@sheridancollege.ca}

\maketitle

\begin{abstract}
Let $f \in M_+(\R_+)$, the class of  nonnegative,  Lebesgure-measurable functions on $\R_+=(0, \infty)$. We deal with integral operators of the form
\[
(T_Kf)(x)=\int_{\R_+}K(x,y)f(y)\, dy, \quad x \in \R_+,
\]
with $K \in M_+(\R_+^2)$.

We are interested in inequalities
\[
\rho_{1}((T_Kf)^*)\leq C\rho_2(f^*),
\]
in which $\rho_1$ and $\rho_2$ are functionals on functions $h \in M_+(\R_+)$,
and
\[
h^*=\mu_h^{-1}(t), \quad t \in \R_+,
\]
where
\[
\mu_h(\lambda)=|\{x \in \R_+: \, h(x)> \lambda\}|.
\]
Specifically, $\rho_1$ and $\rho_2$ are so-called Orlicz-Lorentz functionals of the type
\[
\rho(h)=\rho_{\Phi, u}(h)=\inf\left\{\lambda>0:\, \int_{\R_+}\Phi\left(\frac{h(x)}{\lambda}\right)u(x)\, dx \leq 1\right\}, \quad h \in M_+(\R_+);
\]
here $\Phi(x)=\int_0^x\phi(y)\, dy$, $\phi$ an increasing mapping from $\R_+$ onto itself and $u\in  M_+(\R_+)$.
\medskip

\noindent 2020 Classification: 46B06, 60C05
%

\end{abstract}




\setcounter{page}{1}
\section{Introduction}

Let  $K \in M_+(\R_+^{2})$, the class of nonnegative Lebesgue-measurable functions on $\qquad \qquad\qquad$ $\R_+^{2}=\R_+\times \R_+, \, \R_+=(0, \infty)$.

We consider positive integral operators $T_K$ defined at $f \in M_+(\R^n_+)$ by
\[
(T_Kf)(x)=\int_{\R_+}K(x,y)f(y)\, dy, \quad x \in \R_+.
\]

We are interested in Orlicz gauge functionals $\rho_1$ and $\rho_2$ on $M_+(\R_+)$ for which
\begin{align}\label{1}
\rho_1((T_Kf)^*)\leq C \rho_2(f^*),
\end{align}
where $C>0$ is independent of $f$.

The function $f^*$ in (\ref{1}) is the nonincreasing rearrangement  of $f$ on $\R_+$, with
\[
f^*(t)=\mu_f^{-1}(t),
\]
and
\[
\mu_f(\lambda)=|\{x:\, f(x)>\lambda\}|.
\]

A gauge functional $\rho$ is given in terms of an $N$-function
\[
\Phi(x)=\int_0^x\phi(y)\, dy, \quad x \in \R_+,
\]
$\phi$ being an increasing function mapping $\R_+$ onto itself, {and}  $u$ a locally-integrable (weight) function  in $M_+(\R_+)$. Specifically, the gauge functional $\rho=\rho_{\Phi,u}$ is defined at $f\in M_+(\R_+)$ by
\[
\rho_{\Phi,u}(f)=\inf\{\lambda>0: \int_{\R_+}\Phi\left(\frac{f(x)}{\lambda}\right)u(x)\, dx\leq 1\}.
\]
Thus, in (\ref{1}), $\rho_1=\rho_{\Phi_1,u_1}$ and $\rho_2=\rho_{\Phi_2, u_2}$. The gauge functionals in (\ref{1}) involving rearrangements are referred to as Orlicz-Lorentz functionals.

The inequality (1)
is shown to follow from
\begin{align}\label{3.1}
\rho_1(T_Lf^*)\leq C \rho_2(f^*), \quad f \in M_+(\R_+),
\end{align}
in which $L=L(t,s), \quad s,t \in \R_+$ is the so-called iterated rearrangement of $K(x,y)$, which rearrangement is nonincreasing in each of $s$ and $t$.

Using the concept of the down dual of an Orlicz-Lorentz  gauge functional the inequality (2) is reduced to a gauge functional inequalitiy for general functions in Theorem 2.4. Sufficient conditions for stronger \textit{integral } inequalities like
\begin{align}\label{4.1}
\Phi_1^{-1}\left(\int_{\R_+}\Phi_1(c w(x)(Tf)(x))t(x)\, dx\right)\leq \Phi_2^{-1}\left(\int_{\R_+}\Phi_2(u(y)f(y))v(y)\, dy\right),
\end{align}
with $c>0$ independent of $f\in M_+(\R_+)$ then serve  for (2) and hence for (1).

The integral inequality (\ref{4.1}) is the same as the corresponding gauge functional inequality
\[
\rho_{\Phi_1,t}(w Tf)\leq C \rho_{\Phi_2,v}(u f),
\]
when $\Phi_1(s)=s^q$ and $\Phi_2(s)=s^p, \quad 1<p\leq q < \infty$. The necessary and sufficient conditions are spelled out in this case.



The operators $T_K$ are treated in Section 2. One can prove similar theorems for such operators on $M_+(\R^n)$. We have preferred to work in the context of one dimension where we believe the results are simpler and more elegant.

Section 3 presents examples either illustrating our results or comparing them to previous work. As well we consider integral operators with kernels that are homogeneous of degree $-1$.

In Section 4 a brief historical sketch of our subject is given.

Finally, we remark that results for integral operators with general kernel $K \in M(\R_+)$ can be obtained from ours using the fact that such a $K$ is the difference of two positive kernels, namely, $K=K_+-K_-$, where $K_+(x,y)=\max[K(x,y),0]$ and $K_-(x,y)=\max[-K(x,y), 0]$.

\section{Positive integral operators on $M_+(\R_+)$}

As a first step in our study of (\ref{1})  we focus on the related inequality
\begin{align}\label{9}
\rho_1(T_Kf^*)\leq C \rho_2(f^*), \quad f \in M_+(\R_+).
\end{align}

\begin{theorem}\label{Th6}
Fix $K \in M_+(\R_+^2)$ and let $\Phi_1$ and $\Phi_2$ be $N$-functions, with $\Phi_2(2t)\approx \Phi_2(t)$, $t \gg 1$. Given weight functions $u_1, u_2 \in M_+(\R_+), \, \int_{\R_+}u_2=\infty$, one has (\ref{9}) for $\rho_i=\rho_{\Phi_i, u_i}, \, i=1,2$, if
\begin{align}\label{10}
\rho_{\Psi_2, u_2}(Sg/U_2)\leq C \rho_{\Psi_1,u_1}(g/u_1), \quad g \in M_+(\R+),
\end{align}
\begin{align*}
& (Sg)(x)=\int_0^x(T_K'g)(y)\,dy=\int_0^x\int_0^{\infty}K(z,y)g(z)\, dzdy,\\
& U_2(x)=\int_0^xu_2, \quad \textit{and} \quad \Psi_i(t)=\int_0^t\phi_i^{-1}, \quad i=1,2.
\end{align*}

The functions $\Psi_i$ are the so-called complementary $N$-functions of the $N$-functions $\Phi_i, \, i=1,2$.
\end{theorem}
\begin{proof}
  The identity
  \[
  \int_{\R_+}gT_Kf^*=\int_{\R_+}f^*T'_{K}g
  \]
  readily yields that (\ref{9}) holds if and only if
  \[
  (\rho_2')_d(T_K'g)\leq C \rho_1'(g),
  \]
  where
  \begin{align}\label{11}
  \rho'_1(g)=\rho_{\Psi_1, u_1}(g/u_1)
  \end{align}
  and
  \begin{align}\label{12}
  (\rho'_2)_d(h)=\rho_{\Psi_2, u_2}\left(\int_0^xh/U_2(x)\right), \quad h \in M_+(\R_+).
  \end{align}
  For (\ref{12}), see [3, Theorem 6.2]; (\ref{11}) is straightforward. The proof is complete on taking $h=T'_K(g)$.
\end{proof}

To replace $T_Kf^*$ in (\ref{9}) by $(T_Kf)^*$ we will require
\[
\rho_1\left(t^{-1}\int_0^tf^*\right)\leq C\rho_1(f^*), \quad f \in M_+(\R_+).
\]

For $\rho_1=\rho_{\Phi,u}$ Theorem 2.1 allows one to reduce this inequality to
\[
\rho_{\Psi,u}\left(\int_0^xg/U(x)\right)\leq C \rho_{\Psi, u}(g/u), \quad g \in M_+(\R_+)
\]
and another such inequality; here $\Psi(x)=\int_0^x\phi^{-1}$. As mentioned in the introduction, such gauge functional inequalities are implied by integral inequalities of the form (3). We combine theorems 1.7 and 4.1 from Bloom-Kerman [1] to obtain the next theorem.

\begin{theorem}\label{TH2}
Consider $K(x,y)\in M_+(\R_+^2)$, which, for fixed $y \in \R_+$, increases in $x$ and, for fixed $x\in \R_+$, decreases in $y$, and which satisfies the growth condition
\begin{align}\label{5}
K(x,y)\leq K(x,z)+K(z,y), \quad 0<y<z<x.
\end{align}

Let $t,u,v$ and $w$ be nonnegative, measurable (weight) functions on $\R_+$ and suppose $\Phi_1$ and $\Phi_2$ are $N$-functions having complementary functions $\Psi_1$ and $\Psi_2$, respectively, with $\Phi_1 \circ \Phi_2^{-1}$ convex. Then, there exists $c>0$ such that
\begin{align}\label{6}
\Phi_1^{-1}\left(\int_{\R_+}\Phi_1(cw(x)(T_Kf)(x))t(x)\, dx\right)\leq \Phi_2^{-1}\left(\int_{\R_+}\Phi_2(u(y)f(y))v(y)\, dy\right)
\end{align}
for all $f\in M_+(\R_+)$, if and only if there is a $c>0$, independent of $\lambda, x>0$, with
\begin{align}\label{7}
&\int_0^x\frac{K(x,y)}{u(y)}\phi_2^{-1}\left(\frac{c\alpha(\lambda, x)K(x,y)}{\lambda u(y)v(y)}\right)\, dy\leq c^{-1}\lambda\nonumber\\
&\textit{and}\\
&\int_0^x\frac{1}{u(y)}\phi_2^{-1}\left(\frac{c \beta(\lambda, x)}{\lambda u(y)v(y)}\right)\, dy\leq c^{-1}\lambda\nonumber,
\end{align}
where
\[
\alpha(\lambda, x)=\Phi_2 \circ \Phi_1^{-1}\left(\int_x^{\infty}\Phi_1(\lambda w(y))t(y)\, dy\right)
\]and
\[
\beta(\lambda,x)=\Phi_2 \circ \Phi_1^{-1}\left(\int_x^{\infty}\Phi_1(\lambda w(y)K(y,x))t(y)\, dy\right).
\]
In the case $K(x,y)=\chi_{(0,x)}(y)$ only the first of the conditions in (\ref{7}) is required.
\end{theorem}
\begin{remark} The integral inequality (\ref{6}) with the kernel $K$ of Theorem \ref{TH2} replaced by any $K \in M_+(\R_+^{2m})$ implies the norm inequality
\begin{align}\label{8}
\rho_{\Phi_1,t}(wT_Kf)\leq C \rho_{\Phi_2,v}(uf), \quad f \in M_+(\R_+^n), C>1.
\end{align}
Thus, in the generalization of (\ref{6}), replace $f$ by $\dfrac{f}{C \rho_{\Phi_2,v}(uf)}, \, \dfrac 1C\leq c$, and suppose, without loss of generality, that $\Phi_i(1)=1, \, i=1,2$. Since $\int_{\R_+}\Phi_{2}\left(\dfrac{uf}{\rho_{\Phi_{2,v}}(uf)}\right)v\leq 1,$ we get
\[
\int_{\R_+}\Phi_1\left(\frac{wT_Kf}{C \rho_{\Phi_2, v}(uf)}\right)t\leq 1,
\]
whence (\ref{8}) holds.
\end{remark}

To replace $T_Kf^*$ in (4) by $(T_Kf)^*$ we will require
\[
\rho_1\left(t^{-1}\int_0^tf^*\right)\leq C \rho_1(f^*), \quad f \in M_+(\R_+).
\]
Conditions sufficient for the inequality to hold are given in
\begin{theorem}\label{Th7}
Let $\Phi$ be an $N$-function satisfying $\Phi(2t)\approx \Phi(t), \, t\gg 1$, and suppose $u$ is weight on $\R_+$ with $\int_{\R_+}u=\infty$. Then,
\begin{align*}
\rho_{\Phi,u}\left(t^{-1}\int_0^tf^*\right)\leq C \rho_{\Phi,u}(f^*), \quad f \in M_+(\R_+),
\end{align*}
provided
\begin{align}\label{13}
\phi\left(\frac{c \alpha(\lambda,x)}{\lambda}\right)U(x)\leq c^{-1}\lambda,
\end{align}
with $c>0$ independent of $\lambda, x \in \R_+$, in which
\begin{align}\label{14}
&\alpha(\lambda,x)=\int_x^{\infty}\Phi(\lambda/U(y))u(y)\, dy\nonumber\\
&\textit{and}\nonumber\\
&\int_0^x \phi^{-1}\left(\frac{c\beta(\lambda,x)}{\lambda}\frac{y}{U(y)}\right)\frac{yu(y)}{U(y)}\, dy \leq c^{-1} \lambda,
\end{align}
with $c>0$ independent of $\lambda,x \in \R_+$, in which
\[
\beta(\lambda,x)=\int_x^{\infty}\Phi(\lambda/y)u(y)\, dy.
\]
\end{theorem}
\begin{proof}
In Theorem 2.2, take $K(x,y)=\chi_{(0,x)}(y)/x, \quad \Phi_1=\Phi_2=\Phi$ and $u_1=u_2=u$ to get
\[
(Sg)(x)=\int_0^x\int_y^{\infty}g(z)\frac{dz}{z}=\int_0^xg+x\int_x^{\infty}g(y)\frac{dy}{y},
\]
whence (\ref{10}) reduces to
\begin{align}\label{15}
&\rho_{\Psi,u}\left(\int_0^xg/U(x)\right)\leq C \rho_{\Psi,u}(g/u)\nonumber\\
&\textit{and}\\
&\rho_{\Psi,u}\left(x\int_x^{\infty}g(y)\frac{dy}{y}/U(x)\right)\leq C\rho_{\Psi,u}(g/u)\nonumber.
\end{align}
The condition in (12) is a consequence of the modular inequality
\[
\int_{\R_+}\Psi\left(c\int_0^xg/U(x)\right)u \leq \int_{\R}\Psi(g/u)u,
\]
which, according [1, Theorem 4.1] holds if the first inequality in (\ref{14}) does.

Again, by duality, the condition in (13) holds when the modular inequality
\[
\int_{\R_+}\Phi\left(c \frac 1x \int_0^xg\right)u\leq \int_{\R_+}\Phi\left(g(y)\frac{U(y)}{yu(u)}\right)u(y)\, dy
\]
does, which inequality holds if and only if one has the second condition in (\ref{14}).
\end{proof}

In Theorem \ref{Th8} below we show the boundedness of $T_Kf$ depends on that of $T_Lf^*$, where the kernel $L$ is the iterated rearrangement of $K$ considered in \cite{B}. Thus, for each $x \in \R_+$, we rearrange the function $k_x(y)=K(x,y)$ with respect to $y$ to get $(k^*_x)(s)=K^{*_2}(x,s)=k_s(x)$ and then rearrange the function of $x$ so obtained to arrive at $(K^{*_2})^{*_1}(t,s)=(k_s^*)(t)=L(t,s)$. It is clear from its construction that $K(t,s)$ is nonincreasing in each of $s$ and $t$.

\begin{theorem}\label{Th8}
Consider $K \in M_+(\R_+^2)$ and set $L(t,s)=(K^{*_2})^{*_1}(t,s), \quad s,t \in \R_+$. Suppose $\Phi_1$ and $\Phi_2$ are $N$-functions, with $\Phi_1(2t)\approx \Phi_1(t),\quad t\gg 1$, and let $u_1$ and $u_2$ be weight functions, with $\int_{\R_+}u_1=\infty$. Then, given the conditions (12) and (13) for $\Phi=\Phi_1$ and $u=u_1$ one has
\[
\rho_{\Phi_1,u_1}((T_Kf)^*)\leq C \rho_{\Phi_2,u_2}(f^*),
\]
provided
\[
\rho_{\Phi_1, u_1}(T_Lf^*)\leq C \rho_{\Phi_2,u_2}(f^*), \quad f \in M_+(\R_+).
\]
\end{theorem}
\begin{proof}
We claim
\begin{align}\label{16}
(T_Kf)^{**}(t)\leq (T_Lf^*)^{**}(t), \quad t \in \R_+,
\end{align}
in which, say, $(T_Kf)^{**}(t)=t^{-1}\int_0^t(T_Kf)^*$.

Indeed, given $E \subset \R_+, \quad |E|=t$,
\begin{align*}
\int_ET_Kf&\leq \int_E\int_{\R_+}K^{*_2}(x,s)f^*(s)\,ds\\
&=\int_{\R_+}f^*(s)\, ds\int_{\R_+}\chi_E(x)K^{*_2}(x,s)\, dx\\
&\leq \int_{\R_+}f^*(s)\, ds \int_0^tL(u,s)\, du\\
&=\int_0^t(T_Lf^*)(u)\,du.
\end{align*}
Taking the supremum over all such $E,$ then dividing by $t$ yields (\ref{16}).

Next, the inequality
\[
\rho_{\Phi_1,u_1}((T_Kf)^{**})\leq C \rho_{\Phi_1, u_1}((T_Lf^*)^{**})
\]
is equivalent to
\[
\rho_{\Phi_1, u_1}((T_Kf)^*)\leq C \rho_{\Phi_1, u_1}(T_Lf^*),
\]
given (\ref{14}) for $\Phi=\Phi_1$ and $u=u_1$. For, in that case,
\begin{align*}
\rho_{\Phi_1,u_1}((T_Kf)^*)&\leq \rho_{\Phi_1,u_1}((T_Kf)^{**})\\
&\leq \rho_{\Phi_1,u_1}((T_Lf^*)^{**})\\
&\leq C \rho_{\Phi_1,u}(T_Lf^*).
\end{align*}
The assertion of the theorem now follows.
\end{proof}

\begin{theorem}\label{Th9}
Let $K, L, \Phi_1, \Phi_2, u_1$ and $u_2$ be as in theorem \ref{Th8}. Assume, in addition, that $\Phi_2(2t)\approx\Phi_2(t), \quad t \gg 1,$ $\int_{\R_+}u_2=\infty$ and that conditions (12) and (13) hold for $\Phi=\Phi_1, \, u=u_1$. Then,
\[
\rho_{\Phi_1,u_1}((T_Kf)^*)\leq C \rho_{\Phi_2, u_2}(f^*)
\]
provided
\begin{align}\label{17}
&\rho_{\Psi_2, u_2}(Hg/U_2)\leq C \rho_{\Psi_1, u_1}(g/u_1)\nonumber\\
&\textit{and}\\
&\rho_{\Psi_2, {u}_2}(\bar{M}/U_2 I'g)\leq C \rho_{\Psi_1,{u}_1}({g}/{u}_1), \quad  g  \in M_+(\R_+)\nonumber,
\end{align}
with
\[
(Hg)(x)=\int_0^xM(x,y)g(y)\, dy, \quad (I'g)(x)=\int_x^{\infty}g \quad \textit{and} \quad \bar{M}(x)=\int_0^xL(z,x)\, dz.
\]
\end{theorem}
\begin{proof}
In view of Theorem \ref{Th8} , we need only verify the inequalities (2.1) imply
\begin{align}\label{18}
\rho_{\Phi_1,u_1}(T_Lf^*)\leq C \rho_{\Phi_2, u_2}(f^*), \quad f \in M_+(\R_+).
\end{align}
Now, according to Theorem \ref{Th6}, (\ref{18}) will hold if one has (2.1) with $K=L$. Now,
\begin{align*}
(Sg)(x)&=\int_0^x T'_Lg\\
&=\int_0^{\infty}\int_0^xL(z,y)g(y)\, dy\\
&=\left(\int_0^x+\int_x^{\infty}\right)\int_0^x L(z,y)\, dz g(y)\, dy\\
&=(Hg)(x)+\int_x^{\infty}\int_0^x L(z,y)\, dz g(y)\, dy.
\end{align*}
But, $L(z,y)$ decreases in $y$ for a given $z$, so
\[
\int_x^{\infty}\int_0^xL(z,y)\, dz g(y)\, dy\leq \int_x^{\infty}\int_0^xL(z,x)\, dz f(y)\, dy=\bar{M}(x)(I'f)(x).
\]
Hence, (\ref{17}) implies (\ref{18}).
\end{proof}

We have to this point shown that the inequality (\ref{1}) holds for $\rho_i=\rho_{\Phi_i, u_i, \quad i=1,2}$, whenever the inequalities (\ref{17}) hold. In Theorem \ref{Th10} below we give four conditions which, together with (12) and (13) for $\Phi_1$ and $\Phi_2$, guarantee (\ref{17}).

 The kernel $M(x,y)$ of the operator $H$ is increasing in $x$ for a given $y$ and decreasing in $y$ for a given $x$. Such an operator $H$ is a so-called generalized Hardy operator (GHO) if $M$ satisfies the growth condition
 \begin{align}\label{2.15}
 M(x,y)\leq M(x,z)=M(z,y), \quad y<z<x.
 \end{align}
 This condition is not guaranteed to hold. It has to be assumed in Theorem 2.7 below so that we may apply Theorem 2.2. Theorem 3.1 in the next section gives a class of kernels satisfying (\ref{2.15}).


\begin{theorem}\label{Th10}
Let $K, L, \Phi_1, \Phi_2, u_1, u_2, M, H$ and $\bar{M}$ be as in Theorem \ref{Th9}. Assume, in addition, that $\Phi_1\circ \Phi_2^{-1}$  is convex, $\Phi_1(1)=\Phi_2(1)=1$ and that $M$ satisfies the growth condition (2.15). Then, one has
\begin{align}\label{21}
\rho_{\Phi_1,u_1}((T_Kf)^*)\leq C \rho_{\Phi_2,u_2}(f^*),
\end{align}
with $C>0$ independent of $f \in M_+(\R_+)$, provided there exists $c>0$ independent of $\lambda, x>0$, such that, 
\begin{align}
&\int_0^x M(x,y)u_2(y)\phi_1\left(\frac{c\alpha(\lambda, x)\bar{M}(y)}{\lambda} \right)\, dy\leq c^{-1}\lambda,\notag\\
&\phi_1\left(\frac{c\beta(\lambda,x)}{\lambda}\right)\int_0^xu_2(y)\,dy\leq c^{-1} \lambda,\notag\\
&\textit{and}\\
&\int_0^x \frac{\bar{M}(y) u_2(y)}{U_2(y)}\Phi_2^{-1}\left( \frac{c\gamma(\lambda,x)\bar{M}(y)}{ U_2(y)}\right)\, dy\leq c^{-1}\lambda\notag;
\end{align}
here,
\begin{align*}
&\alpha(\lambda,x)=\Psi_1 \circ \Psi_2^{-1}\left(\int_x^{\infty}\Psi_2\left(\frac{\lambda}{U_2(y)}\right)\right)u_2(y)\,dy,\notag\\
&\beta(\lambda, x)=\Psi_1 \circ \Psi_2^{-1} \int_x^{\infty} \Psi_2\left(\lambda\frac{\bar{M}(x,y)}{U_2(y)}\right)u_2(y)\,dy\\
&\textit{and}\\
&\gamma(\lambda,x)=\Phi_2\circ\Phi_2^{-1}\left(\int_x^{\infty}\Phi_1(\lambda) u_2(y)\, dy\right).\notag
\end{align*}
\end{theorem}
\begin{proof}
Theorem 2.5 tells that (2.16) holds if one has (\ref{18}). Theorem 2.6 reduces (\ref{18}) to the inequalities (\ref{17}) or, equivalently, to the inequalities 
\begin{align}\label{181}
&\rho_{\Psi_2, u_2}(Hg/u2)\leq C \rho_{\Psi_1,u_1}(g/u_1)\notag \\
&\textit{and}\\
&\rho_{\Phi_1,u_1}(If)\leq C \rho_{\Phi_2, u_2}(fU_2/(\bar{M}u_2)), \quad (If)(x)=\int_0^xf\notag.
\end{align}

Remark 2.3 asserts the inequalities (\ref{181}) are consequences of the \textit{integral} inequalities
\begin{align}\label{241}
&\Psi_2^{-1}\left(\int_{\R_+}\Psi_2(c(Hg)(x)/U_2(x))u_2(x)\,dx\right) \leq \Psi_1^{-1}\left(\int_{\R_+}\Psi_1(g(y)/u_1(y))u_1(y)\, dy\right) \notag\\
&\textit{and}\\
&\Phi_1^{-1}\left(\int_{\R_+}\Phi_1(c(If)(x))u_1(x)\,dx\right)\leq \Phi_2^{-1}\left(\int_{\R_+}\Phi_2 \left(\frac{f(y)U_2(y)}{u_2(y)}\right)u_2(y)\, dy\right),\notag
\end{align}  
with $c=1/C$.

To obtain (\ref{241}) first take, in Theorem 2.2, $T_K=H$, $w=1/U_2$, $t=U_2$, $u=1/u_1$ and $v=u_1$ then take $T_K=I$, $w=1$, $t=u_1$, $u=U_2/(\bar{M}u_2)$, $v=u_2$. Theorem 2.2 then yields (\ref{241}) given (2.17).



\end{proof}

\begin{remark}
  One can weaken the third condition in (2.17) by replacing $\bar{M}$ with
  \[
  \bar{\bar{M}}(y)=\sum_{k=0}^{k_0}M(y, 2^{n-k-1})\chi_{(2^{n-k-1}, 2^{n-k})},
  \]
  in which $n \in \Z$ is such that $x\in (2^{n-1}, 2^n)$, $y\in (2^{n-k_0-1}, 2^{n-k_0})$.
\end{remark}

\section{Examples}
\begin{theorem}\label{Th11}
Let $k$ be a nonnegative, nonincreasing function on $\R_+$. Then,  $K(x,y)=k(x+y)$, $x,y \in \R_+$ has its $L(x,y)=k(x+y)$. Moreover, its $M(x,y)=\int_0^xk(y+u)\, du$ satisfies the monotonicity and grawth conditions on the kernel of Theorem 2.2.
\end{theorem}
\begin{proof}
The monotonicity conditions are clearly satisfied. Again, given $y<z<x$, we have
\begin{align*}
M(x,y)&=\int_0^x k(y+u)\, du=\int_0^z k(y+u)\, du+\int_z^xk(y+u)\, du\\
&=M(y,z)+\int_0^{x-z}k(y+z+u)\, du\\
&\leq M(y,z)+\int_0^xk(z+u)\, du\\
&M(y,z)+M(z,x).
\end{align*}
\end{proof}

\begin{theorem}\label{Th12}
Fix the indices $p$ and $q, \, 1<p\leq q< \infty$, and suppose $K(x,y)$ is as in Theorem 3.1. Then,  if $u_1, u_2 \in M_+(\R_+)$ and $\int_{\R_+}u_1=\infty$, one has
\begin{align}\label{24}
\left(\int_{\R_+}(T_Kf)^*(t)^qu_1(t)\, dt\right)^{1/q}\leq C \left(\int_{\R_+}f^*(s)^pu_2(s)\, ds\right)^{1/p},
\end{align}
with $C>0$ independent of $f \in M_+(\R_+)$, if and only if
\begin{align}\label{3.2}
&\alpha(x)^{q-1}\int_0^x{M(x,y)^q}{u_2(y)}\, dy\leq c^{-q}\notag\\
&\beta(x)^{q-1}\int_0^x u_2(y)\, dy\leq c^{-q}\notag\\
&\textit{and}\\
& \gamma(x)^{p'-1}\int_0^x\left(\frac{\bar{M}(y)}{U_2(y)}\right)^{p'}u_2(y)\, dy\leq c^{-p'}, \quad x\in \R_+;\notag
\end{align}
here,
\begin{align*}
&\alpha(x)=\left(\int_x^{\infty}u_2(y)/U_2(y)^{p'}\, dy\right)^{q'/p'},\notag\\
&\beta(x)=\left(\int_x^{\infty}\left(\frac{M(y,x)}{U_2(y)}\right)^{p'} u_2(y)\, dy\right)^{p'/q'}\notag\\
&\textit{and}\\
&\gamma(x)=\left(\int_x^{\infty}u_2(y)\,dy\right)^{p/q}, \quad x \in \R_+.\notag
\end{align*}
\end{theorem}
\begin{proof}
The $N$-functions $\Phi_1(t)=t^q$ and $\Phi_2(t)=t^p,$ as well as the weight $u_1$,  satisfy the conditions required in Theorem \ref{Th10}. According to Theorem \ref{Th11}, so does the kernel $M$  of $H$. We conclude, then, that (2.16) holds for $T_K$, given (2.17), which in our case are the inequality (3.1) and the conditions (3.2). We observe that $\lambda$ cancels out in the latter conditions and we are left with $\alpha_1(1,x)=\alpha(x)$, $\beta_1(1,x)=\beta(x)$ and $\gamma(1,x)=\gamma(x)$.
\end{proof}



We note that if, in Theorem 2.5, the weights $v_i\equiv 1, i=1,2$, then the inequality
$$
\rho_{\Phi_1}((T_Kf)^*)\leq C \rho_{\Phi_2}(f^*)
$$ is the same as
\begin{align}\label{3.4}
\rho_{\Phi_1}(T_Kf)\leq C \rho_{\Phi_2}(f).
\end{align}
In this context we consider an integral operator $T_K$ with $K(x,y)=k(\sqrt{x^2+y^2})$, where $k(t)$ is nonincreasing, with $k(t/2)\leq Ck(t),\, t\in \R_+,$ and $\int_0^ak(\sqrt{x^2+y^2})\, dy< \infty$ for all $a \in \R_+$.

The growth condition on $k$ ensures that
\[
\frac 1C (T_Kf)(x)\leq k(x)\int_0^xf+\int_x^{\infty}kf\leq C(T_Kf)(x),
\]
for $C>1$ independent of $f \in M_+(\R_+)$ and $t\in \R_+$.

Suppose now $\Phi_1(t)=t^q$ and $\Phi_2(t)=t^p, \, 1<p\leq q< \infty$. According to Theorem 2.2, the inequality
\begin{align}\label{in25}
\left[\int_{\R_+}(k(x)\int_0^xf)^q\, dx\right]^{1/q}\leq c\left[\int_{\R_+}f^p\right]^{1/p}
\end{align}
holds if and only if
\[
x \alpha(x)^{p'-1}=\int_0^x \alpha(x)^{p'-1},
\]
where
\[
\alpha(x)=\left[\int_x^{\infty} k(y)^q\, dy\right]^{p/q},
\]
that is,
\begin{align}\label{3.6}
x \left(\int_x^{\infty}k(y)^q\,dy\right)^{p'/q}\leq c^{-p'}, \quad x \in \R_+.
\end{align}

Again,
\[
\left[\int_{\R_+}\left(\int_x^{\infty}\frac{f(y)}{k(y)}\,dy\right)^q\,dx\right]^{1/q}\leq C\left[\int_{\R_+}f^p\right]^{1/p}
\]
if and only if
\[
\left[\int_{\R_+}\left(k(x)\int_0^xg\right)^{p'}\,dx\right]^{1/p'}\leq C \left[\int_{\R_+}g^{q'}\right]^{1/q}
\]
or
\begin{align}\label{in27}
x\left[\int_x^{\infty} k(y)^{p'}\, dy\right]^{q/p'}\leq c^{-q}, \quad x \in \R_+.
\end{align}
Thus, (3.1) holds when and only (3.2) and (3.5) do.

For example, if $K(x,y)=\frac{1}{({x^2+y^2})^{\lambda/2}}, \,\max[1/p', 1/q]<\lambda<1, \, k(x)=x^{-\lambda}$ and the conditions become
\[
x^{1+p'(1/q-\lambda)}\leq c^{-p'}
\]
and
\[
x^{1+q(1/p'-\lambda)}\leq c^{-q},
\]
which conditions are satisfied if and only if $\dfrac 1q=\lambda-\dfrac{1}{p'}$.

This kernel $K$ does not satisfy the classical Kantorovic condition usually invoked to prove (3.1) for $T_K$. Indeed,
\[
K(x,y)\approx {(x+y)^{-\lambda}},
\]
whence, for, $p>1$,
\[
\left[\int_0^{\infty}K(x,y)^{p'}\,dy\right]^{1/p'}\approx x^{-\lambda+1/p'}
\]
and, therefore,
\[
\left[\int_0^{\infty}\left[\int_0^{\infty}K(x,y)^{p'}\,dy\right]^q\, dx\right]^{1/q}\approx \left[\int_0^{\infty}x^{(-\lambda+1/p')q}\, dx\right]^{1/q}=\infty.
\]

Finally, R.~O'Neil in \cite{O} proved that, for $K \in M_+(\R_+^2)$, one has, for each $f \in M_+(\R_+)$,
\[
\frac 1x \int_0^x(T_Kf)^{*}(y)\, dy\leq \int_0^{x}K^*(xy)f^*(y)\, dy.
\]
See also Torchinsky \cite{T}. Given $K(x,y)=k(\sqrt{x^2+y^2})$, as above, $K^*(t)=k(t^{1/2})$, so the right side of the O'Neil inequality is
\[
\int_0^{\infty}k(\sqrt{xy})f^*(y)\, dy.
\]

We see that, if $\Phi_1(2t)\approx\Phi_1(t)$, for $t\gg 1$, the O'Neil inequality yields
\[
\rho_{\Phi_1}((T_Kf)^*)\leq \rho_{\Phi_1}\left(\int_0^{\infty}k(\sqrt{xy})f^*(y)\,dy\right)
\]
Again, we have
\[
\rho_{\Phi_1}((T_Kf)^*)\leq \rho_{\Phi_1}(T_Kf^*)\leq \rho_{\Phi_1}\left(\int_{\R_+}k(\sqrt{x^2+y^2})f^*(y)\, dy\right)
\]
Observing that $k(\sqrt{x^2+y^2})=k\left(\sqrt{\frac{x^2+y^2}{xy}}\sqrt{xy}\right)=k\left(\sqrt{\frac yx+ \frac xy}\sqrt{xy}\right)$, we see our bound is tighter.

\section{Past results}
As mentioned above,when $\Phi_1(t)=t^q, \, \Phi_2(t)=t^p, \, p,q\in [1, \infty]$, and $u_1=u_2=1$, (3.4) becomes the classical Lebesgue inequality
\begin{align}
\left[\int(T_Kf)(x)^q\, dx\right]^{1/q}\leq C \left[\int f(y)^p\, dy\right]^{1/p}.
\end{align}
In the case $K$ is homogeneous of degree $-1$,  that is, $K(\lambda x, \lambda y)=\lambda^{-1}K(x,y), \, \lambda, x, y \in \R_+$, the well-known result of Hardy-Littlewood-Polya in \cite{HLP} asserts that (4.1) holds when $p=q$ if and only if
\[
\int_{\R_+}K(1,y)y^{-1/p}\, dy< \infty.
\]

More generally, Theorem 3.1 of \cite{Boyd} implies

\begin{theorem}\label{last}
Suppose the $N$-function $\Phi$ and the weight $u$ on $\R_+$ satisfy the conditions of Theorem 2.4. Then, the Orlicz-Lorentz functional 
\[
\lambda_{\Phi,u}(f)=\inf \{x>0: \int_{\R_+}\Phi\left(\frac{f^*(x)}{\lambda}\right)u(x)\,dx\leq 1\}
\]
is equivalent to a so-called rearrangement-invariant (r.i.) norm. As such, the dilation operator norm
\[
h_{\Phi,u}(t)=\inf \{M>0: \lambda_{\Phi,u}(f^*(ts))\leq M \lambda_{\Phi,u}(f^*(s)), \quad f \in M(\R_+)\}.
\]
Given $K \in M_+(\R^2_+)$, homogeneous of degree $-1$, one has
\[
\rho_{\Phi,u}(T_Kf)\leq C \rho_{\Phi,u}(f),
\]
with $C>0$ independent of $f\in M_+(\R_+)$, provided
\begin{align*}\label{4.2.1}
\int_{\R_+}K(1,t)h_{\Phi,u}(t)\, dt< \infty.
\end{align*}

Let $\Phi(t)=\Phi_p(t)=t^p, p>1$ and $u(x)=u_{\alpha}(x)=(\alpha+1)x^{\alpha}$, $\alpha>1$. Then,
\[
\rho_{p,u}(t^{-1}\int_0^tf^*)\leq C \rho_{p,u}(f^*),
\]
provided
\[
t^p\int_t^{\infty}s^{\alpha-p}\,ds\leq C \int_0^ts^{\alpha}\,ds,
\]
which, indeed, holds for all $\alpha>-1$.

Again, for $t\in \R_+$,
\[
h_{\Phi,u}(t)=h_{p, \alpha}(t)=\sup_{s>0}\left[\frac{\int_0^{s/t}y^{\alpha}\,dy+(s/t)^p\int_{s/t}^{\infty}y^{\alpha-p}\, dy}{\int_0^sy^{\alpha}\, dy+s^p\int_s^{\infty}y^{\alpha -p}\, dy}\right]\approx t^{-\alpha-1}.
\]
See [6].
\end{theorem}
\begin{proof}
We observe that if $K(1,t)$ decreases in $t$ on $\R_+$, (4.2) will, at least, require $-1<\alpha<0$.
Again, the assumptions on the $N$-functions and weights guarantee that the $\rho_{\Phi_i, u_i}\, i=1,2,$ are equivalent to so-called rearrangement-invariant norms on $M_+(\R_+)$. Theorem 3.1 in \cite{Boyd} then ensures that (4.3) implies (4.2).
\end{proof}


The classical mixed norm condition
\[
\int_{\R_+}\left[\int_{\R_+}K(x,y)^{p'}\, dy\right]^{q/p'}\, dx< \infty, \quad p'=\frac{p}{p-1},
\]
guaranteeing (1) for $T_K$ is due to Kantorovic \cite{Kan}. In Walsh \cite{W} such mixed norm conditions involving so-called weak Lorentz norms followed by interpolation are shown to yield the same inequality. This extended earlier work of Strichartz \cite{S}. For a discussion of yet earlier work of this kind see \cite{KZPS}.

We have already introduced the generalized Hardy operators (GHOs)
\[
(T_Kf)(x)=\int_0^xK(x,y)f(y)\, dy,
\]
studied in Bloom-Kerman \cite{BK}, where $K(x,y)=K(x,y)\chi_{(0,x)}(y)$, with $K$ increasing in $x$, decreasing in $y$ and satisfying the growth condition
\[
K(x,y)\leq K(x,z)+K(z,y), \quad y<z<x.
\]
An exhaustive treatment of these and similar operators on monotone functions for all $p,q \in \R_+$ is given in Gogatishvili-Stepanov \cite{GS}.



\begin{thebibliography}{50}




\bibitem{BK}
S.~Bloom, R.~Kerman, Weighted $L_{\Phi}$ integral inequalities for operators of Hardy type, \emph{Studia Mathematica},  110(1) (1994), 35--52.

\bibitem{B}
A.P~ Bolzinski, Multivariate rearrangements and Banach function spaces with mixed norms, \emph{Trans. Amer. Math. Soc.},  2631(1) (1981), 149--167.

\bibitem{Boyd}
D.~ Boyd, The Hilbert transformation on rearangement invariant Banach spaces, \emph{Thesis, University of Toronto},  (1966).


\bibitem{GKer}
A.~Gogatishvili, R.~Kerman, The rearrangement-invariant space $\Gamma_{p,\Phi}$, \emph{Positivity},  18(2) (2012), 319--345.


\bibitem{GK}
M.L.~Goldman, R.~Kerman, The dual of the cone of decreasing functions in a weighted Orlicz class and the associate of an Orlicz-Lorentz space,  \emph{ Differential Operators. Problems of Mathematical Education, Proc. Intern. Conf. Dedicated to the 75th Birthday of Prof. L. D. Kudrjavtsev} (Moscow, 1998).

\bibitem{GS}
A.~Gogatishvili, V.D.~Stepanov, Reduction theorems for weighted integral inequalities on the cone of monotone functions, \emph{Uspehi Math. Nauk.},  68(4/412) (2013), 3--68.

\bibitem{HLP}
H.G.~Hardy, J.E.~Littlewood, G.~Polya, Inequalities, \emph{Cambridge Univ. Press.}, Ney York, (1952).

\bibitem{Kan}
L.V.~Kantarovic, Integral operators, \emph{Uspehi Math. Nauk.}, 11(2/68) (1956), 3--29.




\bibitem{KZPS}
M.A.~Krasnoselskii, P.P.~Zabreyko, E.I.~Pustylnik, P.E.~Sobolevski, Integral operators in spaces of summable functions,  \emph{Springer} Dordrecht, Netherlands (1976).

\bibitem{O}
R.~O'Neil, Integral trransforms and tensor products on Orlicz spaces and $L(p,q)$ spaces, \emph{J. Analyse}, 21 (1968), 1--276.


\bibitem{S}
R.S.~Strichartz, $L^p$ estimates for integral transforms, \emph{Trans Amer. Math. Sciences}, 126 (1969), 33--50.

\bibitem{T}
A.~Torchinsly, Interpolation of operators and Orlicz classes, \emph{Studia Math.}, 59 (1976), 177--207.

\bibitem{W}
T.~Walsh, On $L^p$ estimates for integral transforms, \emph{Trans. Amer. Math. Soc.}, 155 (1971), 195--215.



\end{thebibliography}
\end{document}